\documentclass[12pt,twoside]{amsart}
\usepackage{amsmath,amsfonts,amssymb,amsthm,enumitem,comment,dirtytalk,tikz-cd,booktabs, tikz-cd,calligra,mathrsfs}

\newtheorem{theorem}{Theorem}[section]

\newtheorem{corollary}[theorem]{Corollary}

\newtheorem{lemma}[theorem]{Lemma}

\newtheorem{remark}[theorem]{Remark}

\newtheorem{question}[theorem]{Question}

\theoremstyle{definition}
\newtheorem{definition}[theorem]{Definition}

\newcommand*{\sheafhom}{\mathscr{H}\kern -.5pt om}
\newcommand*{\sheaftor}{\mathscr{T}\kern -.5pt or}

\dedicatory{Dedicated to David Eisenbud and Francisco Javier Gallego}

\title{Syzygies of adjoint linear series on projective varieties}
\author{Purnaprajna Bangere}\author{Justin Lacini}
\date{}

\address{Department of Mathematics, University of Kansas, 1450 Jayhawk Blvd. , Lawrence, KS 66045, USA}
\email{jlacini@ku.edu}

\address{Department of Mathematics, University of Kansas, 1450 Jayhawk Blvd. , Lawrence, KS 66045, USA}
\email{purna@ku.edu}

\begin{document}

\begin{abstract}
Let $X$ be a smooth complex projective variety of dimension $n$ and let
$A$ be an ample and basepoint free divisor.
We prove $K_X+mA$ satisfies property $N_p$ for $m\geqslant n+1+p$. We also show the graded ring of sections $R(X,K_X+mA)$ is Koszul for $m\geqslant n+2$. 
\end{abstract}

\maketitle

\section{Introduction}\label{sectionintroduction}

Equations defining algebraic varieties have been a topic of interest to geometers for a long time. In the early eighties Mark Green \cite{green1, green2} brought a new perspective to the subject by viewing classical results on projective normality and normal presentation as particular cases of a more general phenomenon involving minimal free resolutions of the homogeneous coordinate ring. Green and Lazarsfeld proved beautiful results for the case of algebraic curves connecting the geometry of the embedding with the structure of the minimal resolution \cite{green1, greenlazarsfeld1, greenlazarsfeld2, greenlazarsfeld3}. 
More recently, much progress has been made in this direction (see for example \cite{voisin1, voisin2}).

\medskip

 A result of M. Green has attracted particular interest as it provides a path for generalizing the syzygy results on curves to higher dimensions. Let $L$ be a line bundle on a curve $C$ of genus $g$ with $\operatorname{deg}(L)\geqslant 2g+1+p$. Then $L$ satisfies property $N_p$, which is defined as follows. Let $S=\operatorname{Sym}^{\bullet} H^0(X, L)$ and consider the graded ring of sections $R=R(X,L)=\oplus_k H^0(X,L^{\otimes k})$ with the natural $S$-module structure. Let $E_{\bullet}$ be a minimal graded free resolution of $R$.

\begin{definition}\label{defnp}
The line bundle $L$ satisfies property $N_p$ if
\begin{enumerate}
    \item $E_0 = S$ if $p\geqslant 0$.
    \item $E_i = S(-i-1)^{\oplus b_i}$ for $1\leqslant i \leqslant p$.
\end{enumerate}
\end{definition}

Green's result therefore shows that a divisor that is as positive as $K_C + mA$ satisfies $N_p$ if $A$ is ample and $m \geqslant p+3$. 
A few years later Reider \cite{reider} proved that $K_S+mA$ is very ample for any $m\geqslant 4$ 
if $A$ is an ample divisor on a smooth algebraic surface $S$. Following Reider's work,
Mukai conjectured that $K_S+mA$ satisfies $N_p$ 
for any $m\geqslant p+4$ if $A$ is ample. Mukai's conjecture is in general open even for $p=0$. Some work has been done in this direction. A stronger version of the conjecture has been proved for anti-canonical rational surfaces in \cite{purna5}, 
a generic version has been proved for surfaces of general type when $A$ is ample and basepoint free in \cite{purna1, purna2} and weaker bounds have been obtained for surfaces with Kodaira dimension zero in \cite{purna1}. For ruled varieties see \cite{butler, park, andreattasommese, purna3} among others.

\medskip

Fujita famously conjectured that
$K_X+mA$ is very ample for any $m\geqslant n+2$ if $A$ is an ample divisor on a projective variety $X$ of dimension $n$. Motivated by these circle of ideas and conjectures, Ein and Lazarsfeld \cite{einlazarsfeld1} proved the following elegant result:
if $A$ is a very ample line bundle on $X$ then $K_X+mA$ satisfies
$N_p$ for any $m\geqslant n+1+p$. It has been an open question ever since (see \cite[Section 4]{einlazarsfeld1}) whether the analogous result holds if $A$ is just ample and basepoint free. 
The purpose of this article is to give a positive answer to this question:

\begin{theorem}\label{intro1}
Let $X$ be a smooth complex projective variety of dimension $n$ and let $A$ be an ample and basepoint free divisor.
Then the line bundle
$L_m =\mathcal{O}_X(K_X+mA)$ satisfies property $N_p$ for any $m\geqslant n+1+p$.
\end{theorem}

Relaxing the positivity assumption on $A$ from very ample to ample and basepoint free poses significant challenges
and our methods are very different from those of \cite{einlazarsfeld1}. Results concerning $N_0$ and $N_1$ were obtained in \cite{debajayan} for varieties with $K_X$ nef with methods similar to \cite{purna2}, and optimal bounds for property $N_p$ have been proved for abelian varieties \cite{pareschiabelian} and Calabi-Yau varieties \cite{purna6,niu}.

\medskip

Whenever property $N_1$ holds, it is a natural question to ask if the homogeneous ring of sections is Koszul. This topic has long been of interest to algebraists and geometers alike. Let $R$ be a graded $\mathbb{C}$-algebra, and let $E_{\bullet}$ be the minimal resolution of $\mathbb{C}$ as an $R$-module. 

\begin{definition}
The ring $R$ is called Koszul if 
$E_i=R(-i)^{\oplus b_i}$ for any $i\geqslant 1$.
\end{definition}

Pareschi \cite{pareschikoszul} adapted the methods of \cite{einlazarsfeld1} to prove that if $A$ is very ample, then $R(X,K_X+mA)$ is Koszul for any $m\geqslant n+2$.
It has been an open question if the same holds when $A$ is just ample and basepoint free. As a corollary of the methods used to prove the main theorem, we obtain the following:

\begin{theorem}\label{intro2}
Let $X$ be a smooth complex projective variety of dimension $n$ and let $A$ be an ample and basepoint free divisor.
Then the graded ring of sections $R(X,K_X+mA)$ is Koszul for any $m\geqslant n+2$.
\end{theorem}

Based on Fujita's conjecture, \cite{einlazarsfeld1} and the present results, we conclude by asking the following
natural question.

\begin{question}\label{quest1}
Let $X$ be a smooth complex projective variety of dimension $n$ and let $A$ be an ample divisor.
Does the line bundle
$L_m =\mathcal{O}_X(K_X+mA)$ satisfy property $N_p$ for all $m\geqslant n+2+p$?
\end{question}

We point out that at the moment this seems
to be completely out of reach, 
and it is not even known 
in the special case when $n=2$ and $p=0$.

\medskip

\textbf{Acknowledgements.} We thank Lawrence Ein for his interest, helpful discussions and encouragement. 

\section{Notation and conventions}

We work over the field of complex numbers $\mathbb{C}$. Of course, everything holds verbatim over any algebraically closed field of characteristic zero.
We use Definition \ref{defnp}
to define the property $N_p$ even when $L$ is just globally
generated. 
A pair $(X,\Delta)$ is the datum of a normal variety $X$ and a $\mathbb{Q}$-Weil divisor $\Delta$ such that $K_X+\Delta$ is $\mathbb{Q}$-Cartier. If $\Delta\geqslant 0$, we say that $(X,\Delta)$ is a log pair. A log resolution of $(X,\Delta)$ is a projective morphism
$f:Y\rightarrow X$ such that $Y$ is smooth and $\operatorname{Ex}(f)\cup \operatorname{Supp}(f^{-1}(\Delta))$ is a simple normal crossings divisor. A $\mathbb{Q}$-Cartier divisor $D$ is nef if $D\cdot C\geqslant 0$ for any curve $C\subseteq X$. If $L_1$ and $L_2$ are two sheaves on $X$, we denote by $L_1\boxtimes L_2$ the sheaf $\operatorname{pr}_1 ^*L_1 \otimes \operatorname{pr}_2^* L_2$ on $X\times X$. We use analogous notation in the case of multiple products and in the case of divisors. 

\section{Preliminaries}

In this section we collect a couple of basic results that we use in the proof
of Theorem \ref{intro1} and Theorem \ref{intro2}. We start with some elementary commutative algebra.

\subsection{Products and diagonals.}\label{algebra}
Let $A$ be a local Noetherian $\mathbb{C}$-algebra
of dimension $n$.
Let $A(r)=A\otimes_{\mathbb{C}}\cdots \otimes_{\mathbb{C}}A$
be the tensor product of $r$ copies of $A$. For
$1\leqslant j \leqslant r$, let $i_j:A\rightarrow A(r)$
be the natural inclusion and let $m_j = i_j(m)\cdot A(r)$.

\begin{lemma}\label{tor1}
Fix $1\leqslant j \neq k \leqslant r$. 
In the above notation, we have:
\begin{enumerate}
    \item $\operatorname{Tor}_i ^{A(r)} (A(r)/m_j, m_k)=0$ for any $i>0$.
    \item $\operatorname{Tor}_0 ^{A(r)}(A(r)/m_j, m_k)=m_k \otimes_{A(r)} A(r)/m_j \cong m_k\cdot A(r)/m_j$.
    \item $\operatorname{Tor}_i ^{A(r)}(A(r)/m_j, A(r)/m_k)=0$
    for any $i>0$.
    \item $\operatorname{Tor}_i ^{A(r)}(m_j, m_k)=0$ for any $i>0$.
\end{enumerate}
\end{lemma}
\begin{proof}
First, notice that 
\[
A(r)/m_j \cong (A/m)\otimes_{\mathbb{C}} A(r-1) \cong A(r-1)
\]

Now let $F_{\bullet}$ be a free resolution of $m$ as an $A$-module. Since $i_k : A\rightarrow A(r)$ is flat,
we have that $F_{\bullet}\otimes_A A(r)$ is a free resolution of $m\otimes_A A(r)\cong m_k$ as an $A(r)$-module. Therefore, we
may compute $\operatorname{Tor}_i ^{A(r)} (A(r)/m_j, m_k)$ by taking the
tensor product with $A(r)/m_j$. In light of the isomorphism
$A(r-1)\cong A(r)/m_j$, we get 
\[
\left( F_{\bullet}\otimes_A A(r)\right) \otimes_{A(r)} A(r)/m_j \cong F_{\bullet}\otimes_A A(r-1)
\]

Since this is exact, we get $(1)$ and $(2)$. 
Consider now the short exact sequence
\[
0\rightarrow m_j\rightarrow A(r) \rightarrow A(r)/m_j \rightarrow 0
\]

Taking tensor products with $A(r)/m_k$ and with $m_k$ gives $(3)$ and $(4)$ respectively. 
\end{proof}

\begin{corollary}\label{ideal1}
In the above notation, we have
\[
m_j\otimes_{A(r)} m_k \cong m_j \cdot m_k = m_j \cap m_k
\]
\end{corollary}
\begin{proof}
Consider once again the short exact sequence
\[
0\rightarrow m_j \rightarrow A(r)\rightarrow A(r)/m_j \rightarrow 0
\]

Then $m_j \cdot m_k = m_j \cap m_k$ follows by taking the tensor
product with $A/m_k$ and using Lemma \ref{tor1} (3), whereas
$m_j \otimes_{A(r)} m_k \cong m_j \cdot m_k$ follows by taking
the tensor product with $m_k$ and using Lemma \ref{tor1} (1). 
\end{proof}

\begin{lemma}\label{localpicture}
Fix $1\leqslant s\leqslant r$ and set 
$m_r(s)=m_1\otimes_{A(r)} \cdots \otimes_{A(r)} m_s$. Then
\begin{enumerate}
    \item $\operatorname{Tor}_i ^{A(r)} (A(r)/m_r(s), m_r)=0$ for all $s\leqslant r-1$ and $i>0$.
    \item $\operatorname{Tor}_i ^{A(r)}(m_r(s), A(r)/m_r)=0$
    for all $s\leqslant r-1$ and $i>0$.
    \item $\operatorname{Tor}_i ^{A(r)} (A(r)/m_r(s), A(r)/m_r)=0$ for all $s\leqslant r-1$ and $i>0$.
    \item $\operatorname{Tor}_i ^{A(r)} (m_r(s), m_r)=0$
    for all $s\leqslant r-1$ and $i>0$.
    \item $m_r(s)\cong m_1\cdots m_s = m_1 \cap \cdots \cap m_s$ for all $s\leqslant r$.
\end{enumerate}
\end{lemma}
\begin{proof}
We prove $(1)-(5)$ by induction on $r$. The case
$r=2$ is settled in Lemma \ref{tor1} and Corollary \ref{ideal1}.
Suppose therefore that the statement holds for $r$, and let us show it holds for $r+1$. If $s\leqslant r-1$, then $(1)-(4)$ follow immediately from the fact that $A(r+1)$ is flat over $A(r)$. Similarly, $(5)$ follows if $s\leqslant r$. Assume now that $s=r$ and consider the following short exact sequence.

\[
0\rightarrow \frac{A(r+1)}{m_{r+1}(r-1) \cap m_r}
\rightarrow \frac{A(r+1)}{m_{r+1}(r-1)}\oplus \frac{A(r+1)}{m_r}
\rightarrow \frac{A(r+1)}{(m_{r+1}(r-1), m_r)}\rightarrow 0 
\]

First notice that we have a natural isomorphism:
\[
\frac{A(r+1)}{(m_{r+1}(r-1), m_r)} \cong \frac{A(r)}{m_r(r-1)}
\]

Furthermore, by $(5)$ we have that
\[
\frac{A(r+1)}{m_{r+1}(r-1) \cap m_r} \cong \frac{A(r+1)}{m_{r+1}(r)}
\]

Therefore, $(1)$ and $(3)$ follow by taking tensor products in the above short exact sequence with $m_{r+1}$ 
and $A(r+1)/m_{r+1}$ respectively, and by using the inductive hypothesis. 

The rest follow as in the proofs
of Lemma \ref{tor1} and Corollary \ref{ideal1}.
\end{proof}

\begin{corollary}\label{resolution}
Let $C_{\bullet}$ be a free resolution of $m=C_0$ as an $A$-module.  
For any $1\leqslant j\leqslant r$, let $C_{\bullet}^j$ be $C_{\bullet}\otimes_A A(r)$ where the module structure is given by the map $i_j:A\rightarrow A(r)$.
Let $C(s)=C^1 \otimes_{A(r)} \cdots \otimes_{A(r)} C^s$, and let
$TC(s)$ be the associated total complex. Then $C(s)$ and
$TC(s)$ are exact.
\end{corollary}
\begin{proof}
Immediate from Lemma \ref{localpicture}.
\end{proof}

Our next goal is to introduce some notation and
globalize Lemma \ref{localpicture}. Let $2\leqslant s\leqslant r$ and $1\leqslant i\neq j\leqslant r$ be positive integers. Let $X$ be a projective variety of dimension $n$.
We denote by 
\[
X(r)=X\times \cdots \times X
\]

the product of $r$ copies of $X$. We denote by 
\[
\operatorname{pr}_i : X(r)\rightarrow X
\]

and by

\[
\operatorname{pr}_{i,j}=\operatorname{pr}_i \times \operatorname{pr}_j : X(r)\rightarrow X\times X
\]

the corresponding projections.
We denote by
\[
\Delta_{X(r)} ^{i,j}=\operatorname{pr}_{i,j}^* \Delta_X
\]

the diagonal relative to the entries $i$ and $j$. We set
\[
\Delta_{X(r)}^1 (s) = \bigcup_{2\leqslant i\leqslant s} \Delta_{X(r)} ^{1,i}
\]

and

\[
\Delta_{X(r)}^2 (s) = \bigcup_{2\leqslant i\leqslant s} \Delta_{X(r)} ^{i-1,i}
\]

Similarly, we set
\[
\mathcal{I}_{\Delta_{X(r)}} ^1 (s)= \mathcal{I}_{\Delta_{X(r)} ^{1, 2}}\otimes 
\mathcal{I}_{\Delta_{X(r)} ^{1, 3}}\otimes \cdots \otimes \mathcal{I}_{\Delta_{X(r)} ^{1, s}}
\]
and

\[
\mathcal{I}_{\Delta_{X(r)}} ^2 (s)= \mathcal{I}_{\Delta_{X(r)} ^{1, 2}} \otimes 
\mathcal{I}_{\Delta_{X(r)} ^{2, 3}} \otimes \cdots \otimes \mathcal{I}_{\Delta_{X(r)} ^{s-1,s}}
\]

Of course,
\[
\mathcal{I}_{\Delta_{X(r)}}^1 (2) = \mathcal{I}_{\Delta_{X(r)}}^2 (2) =
\mathcal{I}_{\Delta_{X(r)}^{1,2}}
\]

\begin{corollary}\label{globalpicture}
Assume that $r\geqslant 3$, $s<t\leqslant r$ and $i>0$. Then:
\begin{enumerate}
    \item $\sheaftor _i ^{\mathcal{O}_{X(r)}}   (\mathcal{O}_{X(r)}/\mathcal{I}_{\Delta_{X(r)}} ^1(s),
    \mathcal{I}_{\Delta_{X(r)} ^{1,t}})=0$.
    \item $\sheaftor _i ^{\mathcal{O}_{X(r)}}
    (\mathcal{I}_{\Delta_{X(r)}} ^1(s), \mathcal{O}_{X(r)}
    /\mathcal{I}_{\Delta_{X(r)} ^{1,t}})=0$.
    \item $\sheaftor _i ^{\mathcal{O}_{X(r)}}
    (\mathcal{O}_{X(r)}/\mathcal{I}_{\Delta_{X(r)}} ^1(s), \mathcal{O}_{X(r)}
    /\mathcal{I}_{\Delta_{X(r)} ^{1,t}})=0$.
    \item $\sheaftor _i ^{\mathcal{O}_{X(r)}}
    (\mathcal{I}_{\Delta_{X(r)}}^1(s), 
    \mathcal{I}_{\Delta_{X(r)} ^{1,t}})=0$.
    \item $\sheaftor _i ^{\mathcal{O}_{X(r)}} (\mathcal{O}_{X(r)}/\mathcal{I}_{\Delta_{X(r)}} ^2(s),
    \mathcal{I}_{\Delta_{X(r)} ^{s,s+1}})=0$.
    \item $\sheaftor _i ^{\mathcal{O}_{X(r)}}
    (\mathcal{I}_{\Delta_{X(r)}} ^2(s), \mathcal{O}_{X(r)}
    /\mathcal{I}_{\Delta_{X(r)} ^{s,s+1}})=0$.
    \item $\sheaftor _i ^{\mathcal{O}_{X(r)}}
    (\mathcal{O}_{X(r)}/\mathcal{I}_{\Delta_{X(r)}} ^2(s), \mathcal{O}_{X(r)}
    /\mathcal{I}_{\Delta_{X(r)} ^{s,s+1}})=0$.
    \item $\sheaftor _i ^{\mathcal{O}_{X(r)}}
    (\mathcal{I}_{\Delta_{X(r)}}^2(s), 
    \mathcal{I}_{\Delta_{X(r)} ^{s,s+1}})=0$.
\end{enumerate}

Furthermore, we have:

\[
\mathcal{I}_{\Delta_{X(r)}} ^1(s)\cong\mathcal{I}_{\Delta_{X(r)} ^{1, 2}}\cdots \mathcal{I}_{\Delta_{X(r)} ^{1, s}}=
\mathcal{I}_{\Delta_{X(r)} ^{1, s}}\cap \cdots\cap \mathcal{I}_{\Delta_{X(r)} ^{1, s}}=\mathcal{I}_{\Delta_{X(r)}^1 (s)}
\]

and
\[
\mathcal{I}_{\Delta_{X(r)}} ^2(s)\cong\mathcal{I}_{\Delta_{X(r)} ^{1, 2}}\cdots \mathcal{I}_{\Delta_{X(r)} ^{s-1, s}}=
\mathcal{I}_{\Delta_{X(r)} ^{1, 2}}\cap \cdots\cap \mathcal{I}_{\Delta_{X(r)} ^{s-1, s}}=\mathcal{I}_{\Delta_{X(r)}^2 (s)}
\]
\end{corollary}
\begin{proof}
Immediate from Lemma \ref{localpicture} after a change of coordinates. 
\end{proof}

\subsection{Koszul cohomology and property $N_p$.} We briefly recall here some
of the basic definitions and properties of Koszul cohomology groups. We refer to Green's original papers \cite{green1} and \cite{green2} and to Lazarsfeld's expository notes \cite{sampling} for a more detailed treatement.
Let $V$ be a finite dimensional vector space, let $S(V)$ be the
symmetric algebra on $V$ and let $B=\oplus_{q\in \mathbb{Z}} B_q$ be a graded $S(V)$-module. Then there is a natural Koszul complex
\[
    \cdots \rightarrow \wedge^{p+1} V \otimes B_{q-1}\xrightarrow{d_{p+1,q-1}} \wedge^p V\otimes B_q
    \xrightarrow{d_{p,q}} \wedge ^{p-1} V\otimes B_{q+1}
    \rightarrow\cdots
\]

We define the Koszul cohomology groups of $B$ as
\[
\mathcal{K}_{p,q}(B,V)=\frac{\operatorname{Ker}(d_{p,q})}{\operatorname{Im}(d_{p+1,q-1})}
\]

Now let $X$ be a smooth projective variety, $L$ a line bundle, $V\subseteq H^0(X,L)$ a vector space and $\mathcal{F}$ a sheaf.
Then we set
\[
\mathcal{K}_{p,q}(X, \mathcal{F}, L, V)=\mathcal{K}_{p,q}(B, V)
\]

where $B=\oplus_{q\in \mathbb{Z}} H^0(X, \mathcal{F}\otimes L^{\otimes q})$. It is common to simply write
$\mathcal{K}_{p,q}(X, \mathcal{F}, L)$ if $V=H^0(X,L)$, and
to write $\mathcal{K}_{p,q}(X, L)$ if furthermore $\mathcal{F}$ is the
structure sheaf. Notice that 
Definition \ref{defnp} may be rephrased as
\[
\operatorname{Tor}_i ^{S(V)}(R, \mathbb{C})_d = 0
\]

for any $0\leqslant i\leqslant p$ and $d\geqslant i+2$. By using the standard Koszul
resolution of $\mathbb{C}$ as an $S(V)$-module and the fact that $\operatorname{Tor}_i ^{S(V)}$ is a balanced functor, we see
that a globally generated line bundle $L$
satisfies property $N_p$ if and only if
\[
\mathcal{K}_{r,q}(X,L)=0
\]

for all $0\leqslant r\leqslant p$ and $q\geqslant 2$. 
With this in mind, in the sequel we will show property $N_p$ via the following theorem.

\begin{theorem}\label{greendiagonal}
Let $X$ be a smooth projective variety, let $L$ be a line bundle and 
let $E$ be a vector bundle. Then 
\[
\mathcal{K}_{p,q}(X,E,L)=0
\]

for all $q\geqslant 2$ provided that
\[
H^1(X(r), \mathcal{I}_{\Delta_{X(r)}} ^1(r)\otimes (E\otimes L^{\otimes k}) \boxtimes L\boxtimes\cdots \boxtimes L)=0
\]

for all $2\leqslant r\leqslant p+2$ and $k\geqslant 1$.
\end{theorem}
\begin{proof}
Follow the proof of \cite[Theorem 3.2]{green2} and use
Corollary \ref{globalpicture} when needed. See also \cite{nori}.
\end{proof}

\begin{remark}
The proof given by Green in \cite{green2} contains a subtle mistake, in that the diagonal sheaves $\mathcal{I}_{\Delta_X} ^2(r)$ are used, rather than the correct $\mathcal{I}_{\Delta_X} ^1(r)$. Unfortunately, this mistake seems to have gone unnoticed in several papers; for example it appears in \cite{inamdar} as well. 
\end{remark}

\subsection{Koszul rings.}\label{koszulrings}
Let $R$ be a positively graded $\mathbb{C}$-algebra.
Let
\[
\cdots \rightarrow P_{n+1}\xrightarrow{\phi_{n+1}}
P_n \xrightarrow{\phi_n}\cdots\rightarrow P_1
\xrightarrow{\phi_1} P_0\rightarrow \mathbb{C}\rightarrow 0
\]

be a minimal graded free resolution of $\mathbb{C}$. Notice that this resolution is finite if and only if $R$ is a polynomial ring. The ring $R$ is called Koszul if
all the entries in the matrix representing $\phi_n$ have degree one. This is equivalent to requiring that
$\operatorname{Tor}_i ^R(\mathbb{C}, \mathbb{C})_d=0$
for all $i>0$ and $d\neq i$.
We will use the following criterion.

\begin{theorem}\label{mehtadiagonal}
Let $X$ be a projective variety and let $L$ be a line bundle.
The graded ring of sections $R(X,L)$ is Koszul if
\[
H^1(X(r), \mathcal{I}_{\Delta_{X(r)}}^2(r)\otimes L\boxtimes L\boxtimes \cdots \boxtimes L^{\otimes k})=0
\]

and

\[
H^1(X(r), \mathcal{I}_{\Delta_{X(r)} ^{r-1, r}}\otimes L\boxtimes L\boxtimes \cdots \boxtimes L^{\otimes k})=0
\]

for all $r\geqslant 2$ and $k\geqslant 1$.
\end{theorem}
\begin{proof}
This is Proposition 1.9 in \cite{mehtainamdar}.
\end{proof}

\subsection{Resolutions of the diagonal.}
Let $M=\Omega_{\mathbb{P}^n}(1)$.
Beilinson's resolution of the diagonal $\Delta_{\mathbb{P}^n}$ of $\mathbb{P}^n\times \mathbb{P}^n$ is the exact sequence:
\begin{equation*}\tag{1}
0\rightarrow \wedge^n (\mathcal{O}_{\mathbb{P}^n}(-1)\boxtimes M) \rightarrow
\cdots 
\rightarrow \mathcal{O}_{\mathbb{P}^n}(-1)\boxtimes M\rightarrow \mathcal{I}_{\Delta_{\mathbb{P}^n}}\rightarrow 0
\end{equation*}

This resolution will play a crucial role in the proof of
Theorem \ref{intro1}.

\subsection{Duality for finite morphisms.}\label{duality}
We recall here 
a well-known duality statement.
Let $f:Y\rightarrow X$ be a finite surjective morphism of normal projective
schemes. Let $\omega_X$ be a dualizing sheaf for $X$
and assume that $\omega_X$ is invertible. 
By \cite[Chapter III, Exercise 7.2 (a)]{ag}, we have that 
$\omega_Y = f^! \omega_X = \sheafhom _X(f_* \mathcal{O}_Y, \omega_X)$ is a dualizing sheaf for $Y$ (see also \cite[Chapter III, Exercise 6.10]{ag}).
Therefore we get
\[
f_* \omega_Y \cong \sheafhom_X(f_* \mathcal{O}_Y, \omega_X)
\]

so that
\[
f_* \omega_{Y/X}\cong (f_* \mathcal{O}_Y)^*
\]

Since $X$ and $Y$ are normal, we have that $f_* \mathcal{O}_Y$ is split by the trace morphism.  In particular,
the structure sheaf $\mathcal{O}_X$ is canonically a direct summand of
$f_* \omega_{Y/X}$.

\section{The very ample case}

In this section we present a first simple approach to the problem in the special case where $A$ is very ample.
Our hope is that this preliminary case may serve as a motivating example, and highlight some of the main difficulties of the problem, but we do not aim here for optimal bounds. Since the rest of the discussion does not rely on the methods of this section, the reader who wishes to do so may directly skip to Section \ref{theproof}.

Recall that our aim is to show property $N_p$ and property Koszul via Theorem \ref{greendiagonal} and Theorem \ref{mehtadiagonal} respectively. 
A first natural approach is then to try using well known generalizations of Kodaira's vanishing theorem (i.e. Kawamata-Viehweg and Nadel vanishing theorems).
Let us briefly recall the definition of multiplier ideals.

\begin{definition}
Let $(X,\Delta)$ be a log pair with $X$ smooth, and let $\mu : Y\rightarrow X$ be a log resolution. We define the multiplier ideal sheaf of $\Delta$ to be
\[
   \mathcal{I}(X,\Delta) = \mu _* \mathcal{O}_Y (K_{Y/X} - \lfloor \mu ^* \Delta \rfloor) \subseteq \mathcal{O}_X
\]
\end{definition}

\begin{theorem}[Nadel's vanishing theorem]
Let $X$ be a smooth complex projective variety and $\Delta\geqslant 0$ a $\mathbb{Q}$-divisor on $X$. Let $L$ be any integral divisor such that $L-\Delta$ is big 
and nef. Then 
\[
H^i (X, \mathcal{O}_X (K_X + L)\otimes \mathcal{I}(X,\Delta))=0
\]

for $i>0$.
\end{theorem}

The above vanishing closely resembles the one needed in Theorem \ref{greendiagonal} and \ref{mehtadiagonal}. In fact, one only has to arrange for the given diagonal ideals to be multiplier ideals of appropriate divisors. To this end, let $M_1, \cdots, M_{n+1}$ be general divisors
in 
\[
\wedge^2 H^0(X,A)\subseteq H^0(X,A)\otimes H^0(X,A)\cong H^0(X\times X, A\boxtimes A)
\]

and let $M=(M_1 + \cdots + M_{n+1})/(n+1)$. An easy check shows that
\[
\mathcal{I}(X\times X, n\cdot M)=\mathcal{I}_{\Delta_X}
\]

Furthermore, if we take 
\[
M(r)=\sum_{j=2} ^{r} \operatorname{pr}_{1,j} ^* M
\]

we see that 
\[
\mathcal{I}(X(r), n\cdot M(r))=\mathcal{I}_{\Delta_{X(r)}} ^1 (r)
\]

By Nadel's vanishing theorem and by Theorem \ref{greendiagonal}, we get:

\begin{theorem}\label{veryamplenp}
Let $X$ be a smooth projective variety of dimension $n$, let $A$ be a very ample divisor and let $B$ be a nef divisor. 
Set $L_m = \mathcal{O}_X(K_X+mA+B)$ for some $m\geqslant 0$.
Then $L_m$ satisfies property $N_p$ if $m\geqslant np+n+1$.
\end{theorem}

By taking

\[
M(r)=\sum_{j=2} ^{r} \operatorname{pr}_{j-1,j} ^* M
\]

a similar discussion shows:

\begin{theorem}\label{veryamplekoszul}
Let $X$ be a smooth projective variety of dimension $n$, let $A$ be a very ample divisor and let $B$ be a nef divisor. 
Set $L_m = \mathcal{O}_X(K_X+mA+B)$ for some $m\geqslant 0$.
Then $R(X,L_m)$ is Koszul if $m\geqslant 2n+1$.
\end{theorem}

Theorem \ref{veryamplenp} and Theorem \ref{veryamplekoszul} recover the main results of \cite{einlazarsfeld1} and \cite{pareschikoszul} respectively, with weaker bounds. We will show below how to extend these results to the case in which $A$ is only ample and basepoint free and how to strengthen them to optimal bounds. 
For the moment, let us just remark that the above method no longer works. A first problem is that we can no longer hope to cut down to the diagonal $\Delta_X$ by using sections in $\wedge^2 H^0(X,A)$. 
If we try to do this, we instead get an \say{enlarged diagonal}, say
$\Gamma_X$. A more serious problem is that
the associated multiplier ideals give rise to 
non-reduced schemes.
In fact, by Skoda's theorem \cite[Theorem 9.6.21]{positivity2}, we have that

\[
\mathcal{I}(X\times X, n\cdot M)=\mathcal{I}_{\Gamma_X}\cdot
\mathcal{I}(X\times X, (n-1)\cdot M)
\]

In general, however, 
\[
\mathcal{I}(X\times X, (n-1)\cdot M)\neq \mathcal{O}_{X\times X}
\]

and therefore it is not clear how to relate 
$\mathcal{I}(X\times X, n\cdot M)$ with $\mathcal{I}_{\Delta_X}$.
We show below how to circumvent both problems via certain resolutions of the diagonal and duality theory. 

\section{Proof}\label{theproof}

We are now ready to start the proof of Theorem \ref{intro1}. 
For the reader's convenience, we divide the proof in several smaller steps. 

\subsection{The setup.}\label{setup} Let $X$ be a smooth projective variety of dimension $n$ and let $A$ be an ample and basepoint free divisor. Let $V\subseteq H^0(X, A)$ be a basepoint free subspace of dimension $n+1$. Let $\alpha : X\rightarrow \mathbb{P}^n$ be the corresponding finite flat morphism. 
Let $\beta: Y\rightarrow X$ be the Galois closure of $\alpha$
and let $\delta=\alpha\circ \beta$.
Let $G=\operatorname{Gal}(Y/\mathbb{P}^n)$ and $H=\operatorname{Gal}(Y/X)$.
As usual, we denote by $\Delta_X\subseteq X\times X$ and $\Delta_Y\subseteq Y\times Y$ the respective diagonals. 
Let $\Gamma_Y$ and $\Gamma_X$
be the schemes
defined by $\mathcal{I}_{\Delta_{\mathbb{P}^n}}\cdot \mathcal{O}_{Y\times Y}$
and $\mathcal{I}_{\Delta_{\mathbb{P}^n}}\cdot \mathcal{O}_{X\times X}$ respectively. 
We have $\Gamma_Y = Y\times_{\mathbb{P}^n} Y$ and
$\Gamma_X = X\times_{\mathbb{P}^n} X$.
We set $\mathcal{I}_{\Gamma_{Y(r)}}^j (s)=\mathcal{I}_{\Delta_{\mathbb{P}(r)}}^j(s)\cdot \mathcal{O}_{Y(r)}$ and 
$\mathcal{I}_{\Gamma_{X(r)}}^j(s)=\mathcal{I}_{\Delta_{\mathbb{P}(r)}}^j (s)\cdot \mathcal{O}_{X(r)}$ for $j=1,2$, and we denote
by $\Gamma_{Y(r)}^j(s)$ the corresponding schemes. Finally, we denote by $\Delta_{Y(r)}^j(s)$
and $\Delta_{X(r)}^j(s)$ the schemes defined by
$\mathcal{I}_{\Delta_{Y(r)}}^j (s)$ and 
$\mathcal{I}_{\Delta_{X(r)}}^j (s)$ for $j=0,1$.

\subsection{Enlarged diagonals.} 
Here we use Beilinson's resolution $(1)$ to get resolutions of the sheaves $\mathcal{I}_{\Delta_{\mathbb{P}^n(r)}}^j(s)$. Although these 
\say{resolutions} do not consist of locally free sheaves, they will prove to be equally
useful for our purposes. We start our study by pulling back
Beilinson's resolution $(1)$ to $\mathbb{P}^n (r)$ via $\operatorname{pr}_{j_1,j_2}$ for $1\leqslant j_1<j_2\leqslant r$ to get
the following exact sequences
\begin{align*}
0\rightarrow \operatorname{pr}_{j_1, j_2} ^* \wedge^n (\mathcal{O}_{\mathbb{P}^n} (-1)\boxtimes M)\rightarrow \cdots &\rightarrow \operatorname{pr}_{j_1,j_2} ^* \mathcal{O}_{\mathbb{P}^n} (-1)\boxtimes M \rightarrow\\
&\rightarrow \tag{$j_1.j_2$}
\mathcal{I}_{\Delta_{\mathbb{P}^n} ^{j_1,j_2}}\rightarrow 0
\end{align*}

We consider each sequence $(j_1.j_2)$ as an exact complex positively graded
and with the ideal sheaf in degree zero, which we call $C(j_1, j_2, n, r)$.
Let
\[
C^1(s, n, r)=C(1,2,n,r)\otimes C(1,3,n,r)\otimes \cdots
\otimes C(1, s, n, r)
\]

and

\[
C^2(s, n, r)=C(1,2,n,r)\otimes C(2,3,n,r)\otimes \cdots
\otimes C(s-1, s, n, r)
\]

Let $TC^1(s,n,r)$ and $TC^2(s,n,r)$ be the corresponding total complexes. An application of Corollary \ref{resolution} gives:

\begin{lemma}\label{exact}
Let $0<s\leqslant r$ and $n$ be positive integers.
Then $C^1(s,n,r)$, $C^2(s,n,r)$, $TC^1(s,n,r)$ and $TC^2(s,n,r)$ are exact.
\end{lemma}

Write $M_V=\alpha^* M$.
Pulling back Beilinson's resolution $(1)$ to $X$ via $\alpha$ gives a resolution of $\Gamma_X$:

\begin{align*}
0\rightarrow \wedge^n (\mathcal{O}_X(-A)\boxtimes M_V)
\rightarrow \cdots 
&\rightarrow \mathcal{O}_X(-A)\boxtimes M_V\rightarrow \\
& \rightarrow \tag{2}
\mathcal{I}_{\Gamma_X}\rightarrow 0
\end{align*}

Similarly, pulling back the complexes $C(j_1, j_2, n, r)$, $C^1(s,n,r)$ 
and $C^2(s,n,r)$ to $X(r)$ gives exact complexes, which we denote by $C(X, j_1, j_2, r)$, $C^1(X,s,r)$ and $C^2(X,s,r)$ respectively. 

\subsection{Vanishing.} Here we prove the crucial vanishing results in view of Theorem \ref{greendiagonal}.
\begin{lemma}\label{vanishing1}
Let $L_m=\mathcal{O}_{X} (K_X+mA)$ for some $m\geqslant n+1$.
Then
\[
H^i (X, L_m ^{\otimes k} \otimes \mathcal{O}_X(lA))=0
\]
for any $k\geqslant 1$, $m+l\geqslant 1$, and $1\leqslant i\leqslant n$.
\end{lemma}
\begin{proof}
Immediate application of Kodaira's vanishing theorem.
\end{proof}

\begin{lemma}\label{vanishing2}
Let $L_m =\mathcal{O}_{X} (K_X+mA)$
for some $m\geqslant n+1$. Then
\[
H^i(X, \wedge^ j M_V \otimes L_m ^{\otimes k}\otimes\mathcal{O}_X(lA))=0
\]

if $k\geqslant 1$ and either 

\begin{enumerate}
    \item $i>j$ and $m+l\geqslant 1$, or
    \item $i\geqslant 1$ and $m+l\geqslant n+2-j$.
\end{enumerate}
\end{lemma}
\begin{proof}
For $(1)$, we proceed by induction on $j$, starting with the case $j=1$. Consider the short exact sequence
\[
0\rightarrow M_V \rightarrow V\otimes_{\mathbb{C}}
\mathcal{O}_X\rightarrow \mathcal{O}_X(A)\rightarrow 0
\]

Then the result follows immediately from Lemma \ref{vanishing1}. Suppose therefore that the result holds for $j-1$ and let us prove it for $j$. Consider the short exact sequence

\begin{equation*}\tag{$*$}
0\rightarrow \wedge^j M_V \rightarrow 
\wedge^j V\otimes_{\mathbb{C}} \mathcal{O}_X
\rightarrow \wedge^{j-1} M_V \otimes
\mathcal{O}_X(A)\rightarrow 0    
\end{equation*}

The statement follows then immediately by the inductive hypothesis. Now we prove $(2)$ by descending induction on $j$, starting with the case $j=n$. In this case $\wedge^n M_V\cong \mathcal{O}_X(-A)$ and the statement follows from Kodaira's vanishing theorem. Suppose then that the $(2)$ holds for $j$, and let us prove it for $j-1$. Consider again the short exact sequence $(*)$. Then statement follows then by Lemma \ref{vanishing1} and
the inductive hypothesis. 
\end{proof}

\begin{lemma}\label{vanishing3}
Fix $r\geqslant 2$.
Let $L_m =\mathcal{O}_{X} (K_X+mA)$
for some $m\geqslant n+r-1$.
Let $i_t\geqslant 0$ for $1\leqslant t\leqslant r$ and $a_t \geqslant 1$ for $1\leqslant t\leqslant r-1$ be two sequences of positive integers. Let $\sum_{t=1} ^r i_t = i$ and $\sum_{t=1} ^{r-1} a_t = d$. Assume that $i\geqslant d-r+2$.
Then
\[
 H^{i_1}(X, \mathcal{O}_X(-dA)\otimes L_m ^{\otimes k}) \otimes \bigotimes_{t=2} ^r H^{i_t}(X, \wedge^{a_{t-1}} M_V \otimes L_m)=0
\]

for any $k\geqslant 1$.
\end{lemma}
\begin{proof}
If $i_t>0$ for any $t\geqslant 2$, then we are done by Lemma \ref{vanishing2} $(2)$. Therefore we may assume that $i_1 = i$.
Since $d\geqslant r-1$, we have that $i_1\geqslant 1$. 
If $i_1>n$, then there is nothing to prove, so we may assume that $i_1\leqslant n$. Therefore,
\[
d\leqslant n+r-2
\]

and we may conclude by Lemma \ref{vanishing1}.
\end{proof}

\begin{lemma}\label{vanishing4}
Fix $r\geqslant 1$.
Let $L_m =\mathcal{O}_{X} (K_X+mA)$
for some $m\geqslant n+2$.
Let $i_t\geqslant 1$ for $1\leqslant t\leqslant r$ and $0\leqslant a_t \leqslant n$ for $0\leqslant t\leqslant r$ be two sequences of positive integers. 
Let $\sum_{t=1} ^r i_t = i$ and $\sum_{t=0} ^r a_t = d$.
Assume that $i\geqslant d-r$.
Then
\[
 \bigotimes_{t=1} ^r H^{i_t}(X, \wedge^{a_{t-1}} M_V\otimes L_m ^{\otimes k_t}\otimes \mathcal{O}_X(-a_t A))=0
\]

for any $k_t\geqslant 1$.
\end{lemma}
\begin{proof}
If $i_t>a_{t-1}$ for any $1\leqslant t\leqslant r$, then we are done by Lemma \ref{vanishing2} $(1)$. Therefore, we may assume that $i_t\leqslant a_{t-1}$ for all $1\leqslant t\leqslant r$.
Similarly, if $a_{t}\leqslant a_{t-1}$ for any $1\leqslant t\leqslant r$, then we are done by Lemma \ref{vanishing2} $(2)$. 
Therefore, we may assume that $a_{t}>a_{t-1}$ for all $1\leqslant t\leqslant r$. 
In particular, $a_r\geqslant r+1$.
Putting everything together, we have
\[
\sum_{t=1}^{r} a_{t-1}\geqslant \sum_{t=1} ^{r} i_{t}
\geqslant \sum_{t=0} ^r a_t -r
\]

Therefore, we get $a_r\leqslant r$, contradiction.
\end{proof}

\begin{lemma}\label{vanishing4bis}
Fix $r\geqslant 2$.
Let $L_m =\mathcal{O}_{X} (K_X+mA)$
for some $m\geqslant n+2$.
Let $i_t\geqslant 0$ for $1\leqslant t\leqslant r$ and $a_t \geqslant 1$ for $1\leqslant t\leqslant r-1$ be two sequences of positive integers. 
Choose two integers $0\leqslant a_0\leqslant n$ and $0\leqslant a_r\leqslant n$. 
Let $\sum_{t=1} ^r i_t = i$ and $\sum_{t=0} ^r a_t = d$.
Let $z$ be the cardinality of the set $\{t\in\{0,r\}|a_t>0\}$.
Assume that $i\geqslant d-r-z+2$.
Then
\[
 \bigotimes_{t=1} ^r H^{i_t}(X, \wedge^{a_{t-1}} M_V\otimes L_m ^{\otimes k_t}\otimes \mathcal{O}_X(-a_t A))=0
\]

for any $k_t\geqslant 1$.
\end{lemma}
\begin{proof}
We may assume that $a_t\leqslant n$ for all $0\leqslant t\leqslant n$.
Let $s$ be the cardinality of the set $\{1\leqslant t\leqslant r| i_t=0\}$. We proceed by induction on $s$. The case $s=0$ is settled in Lemma \ref{vanishing4}.
Assume therefore that the statement holds for $s$ and let us show it for $s+1$.  
Let $t_0=\operatorname{min}\{t|i_t=0\}$. 
If $t_0=1$ or $t_0=r$, then 
the statement follows by the inductive hypothesis after discarding the index $t_0$. Assume therefore that $1<t_0<r$.
Let $j=\sum_{t=1} ^{t_0-1} i_t$ and $e=\sum_{t=0} ^{t_0-1} a_t$.
If $j\geqslant e - t_0 + 1$, then we are done by Lemma \ref{vanishing4}.
Therefore we may assume that $i-j \geqslant (d-e) - (r-t_0) - z + 2$.
Since $z$ can only increase after cuts, we are done by inductive hypothesis. 
\end{proof}

\begin{lemma}\label{vanishing5}
Fix $r\geqslant 2$.
Let $L_m =\mathcal{O}_{X} (K_X+mA)$ for some $m\geqslant 0$.
Let $a_t \geqslant 0$ for $1\leqslant t\leqslant r-1$ be a sequence of integers and set
$\sum_{t=1} ^{r-1} a_t = d$. Let $s$ be the cardinality of the set $\{t | a_t=0\}$. 
Assume that $i\geqslant d+s-r+2$.
If $m\geqslant n+2$ and $k\geqslant 1$, then
\[
H^i \left(X(r), \bigotimes_{t=1} ^{r-1} C(X,t,t+1,r)_{a_t} \otimes L_m 
\boxtimes L_m \cdots \boxtimes L_m^{\otimes k}\right)=0
\]

If $m\geqslant n+r-1$ and $k\geqslant 1$ instead, then
\[
H^i \left(X(r), \bigotimes_{t=1} ^{r-1} C(X,1,t+1,r)_{a_t} \otimes 
L_m ^{\otimes k}\boxtimes L_m \cdots
\boxtimes L_m\right)=0
\]
\end{lemma}
\begin{proof}
We proceed by induction on $s$. The case $s=0$ follows immediately from Lemma \ref{vanishing3}, Lemma \ref{vanishing4bis} and K\"unneth's formula. Suppose therefore that the statement holds for $s$, and let us show it for $s+1$. Let $t_0$ be an index such that $a_{t_0}=0$. By Lemma \ref{exact}, it is then enough to show that
\[
H^{i+j'}\left(X(r), C(X,t_0,t_0+1, r)_j \otimes
\bigotimes_{t\neq t_0} C(X,t,t+1,r)_{a_t} \otimes L_m 
\boxtimes
 \cdots \boxtimes L_m^{\otimes k}\right)=0
\]

and

\[
H^{i+j'}\left(X(r), C(X,1,t_0+1, r)_j\otimes \bigotimes_{t\neq t_0} C(X,1,t+1,r)_{a_t} \otimes L_m ^{\otimes k}\boxtimes\cdots
\boxtimes L_m\right)=0
\]

for any $j\geqslant 1$ and $j'\geqslant j-1$.
This holds by inductive hypothesis, and we are done. 
\end{proof}

\begin{theorem}\label{mainvanishing}
Fix $r\geqslant 2$.
Let $L_m =\mathcal{O}_{X} (K_X+mA)$
for some $m\geqslant 0$. 
If $m\geqslant n+2$, then
\[
H^1(X(r), \mathcal{I}_{\Gamma_{X(r)}}^2(r)\otimes L_m \boxtimes L_m\boxtimes \cdots
\boxtimes L_m ^{\otimes k})=0
\]

for any $k\geqslant 1$. If $m\geqslant n+r-1$ instead, then
\[
H^1(X(r), \mathcal{I}_{\Gamma_{X(r)}}^1(r)\otimes L_m ^{\otimes k}\boxtimes L_m\boxtimes \cdots
\boxtimes L_m)=0
\]

for any $k\geqslant 1$.
\end{theorem}
\begin{proof}
Immediate consequence of Lemma \ref{vanishing5} applied to the complexes $C^1(X,r,r)$ and $C^2(X,r,r)$.
\end{proof}

\subsection{Projective normality.}\label{subsectionnormality}
Here we prove:

\begin{theorem}\label{projectivenormality}
Let $X$ be a smooth complex projective variety of dimension $n$ and let $A$ be an ample and basepoint free divisor.
Then the line bundle
$L_m =\mathcal{O}_X(K_X+mA)$ satisfies property $N_0$ for any $m\geqslant n+1$. 
\end{theorem}

We fix notation at in Subsection \ref{setup}. First, notice that $\Gamma_X$ is reduced since $\alpha$
is flat, $\mathbb{P}^n$ is reduced and $\Gamma_X$ is generically reduced. Furthermore, $\Gamma_X$ is Cohen-Macaulay since $\alpha$ is flat and $\mathbb{P}^n$ is smooth.
Let $\omega_Y$ be the dualizing sheaf of $Y$. 
Let $Y_0$ be the smooth locus of $Y$. Notice that $Y\setminus Y_0$ has codimension at least two since $Y$ is normal; 
in particular
$\omega_Y \cong i_* \omega_{Y_0}$. Up to shrinking $Y_0$
while keeping the codimension at least two,
we may assume that $X_0=\beta(Y_0)$ is open and 
$\beta:Y_0 \rightarrow X_0$ is finite and flat.
Let $\Gamma_{Y_0}=Y_0 \times_{\mathbb{P}_n} Y_0$.
Consider
\[
F=(\omega_Y \otimes \beta^* \mathcal{O}_X (mA))
\boxtimes
(\omega_Y \otimes \beta^* \mathcal{O}_X (mA))
\]

There is an injective map
\[
F|_{\Gamma_{Y_0}}
\rightarrow
\bigoplus_{g\in G} F|_{g^* \Delta_{Y_0}}
\]

where we let $G$ act on the second entry. By taking
invariants we get an isomorphism
\[
\left(F|_{\Gamma_{Y_0}}\right)^G
\cong
\left(\bigoplus_{g\in G} F|_{g^* \Delta_{Y_0}}\right)^G
\]

Therefore we have a surjective map
\[
H^0(\Gamma_{Y_0}, F|_{\Gamma_{Y_0}})
\rightarrow
H^0(\Delta_{Y_0} , F|_{\Delta_{Y_0}})
\]

Equivalently, this also follows by noticing that the ramification formula applied to the morphism $Y_0\rightarrow \delta(Y_0)$ gives rise to a canonical isomorphism

\[
\operatorname{pr}_1 ^* (\omega_Y \otimes \beta^* \mathcal{O}_{X}(mA)|_{\Gamma_{Y_0}}
\cong 
\operatorname{pr}_2 ^* (\omega_Y \otimes \beta^* \mathcal{O}_{X}(mA)|_{\Gamma_{Y_0}}
\]

\textbf{The Galois case.} If $X$ is Galois over $\mathbb{P}^n$, then
$Y=X$. Therefore Lemma \ref{finiterestriction} below follows immediately.

For the general case, consider the trace map 

\[
\operatorname{Tr}: (\beta\times \beta)_* F
\rightarrow L_m  \boxtimes L_m
\]

We have the following diagram, which commutes up to multiplication by $|H|$ on the sections pulled back 
from $H^0(\Delta_{Y_0}, F|_{\Delta_{Y_0}})$ 
to $H^0(\Gamma_{Y_0}, F|_{\Gamma_{Y_0}})$ via the Galois action.

\begin{tikzcd}
    H^0(\Gamma_{X_0}, (\beta\times\beta)_* F|_{\Gamma_{X_0}})\arrow{r}{\operatorname{Tr}|_{\Gamma_{X_0}}}\arrow{d}{r} & H^0(\Gamma_{X_0}, L_m\boxtimes L_m\arrow{d}{r})\\
    H^0(\Delta_{X_0}, (\beta\times\beta)_* F|_{\Delta_{X_0}}) 
    \arrow{r}{\operatorname{Tr}|_{\Delta_{X_0}}}& 
    H^0(\Delta_{X_0}, L_m \boxtimes L_m)
\end{tikzcd}

The left vertical map and the trace maps are both surjective. 
Therefore the right vertical map also is surjective. 
Since $\Delta_X$ and 
$\Gamma_X$ are both reduced and Cohen-Macaulay,
we have proved:

\begin{lemma}\label{finiterestriction}
The restriction map
\[
H^0(\Gamma_X, L_m ^{\otimes k} \boxtimes L_m)\rightarrow H^0(\Delta_X, L_m ^{\otimes k} \boxtimes L_m)
\]

is surjective. 
\end{lemma}

Theorem \ref{projectivenormality} then follows by combining
Lemma \ref{finiterestriction} and Theorem \ref{mainvanishing}.

\subsection{Property $N_p$ and Koszul.} 
Here we prove:
\begin{theorem}\label{maintheorem1}
Let $X$ be a smooth complex projective variety of dimension $n$ and let $A$ be an ample and basepoint free divisor.
Then the line bundle
$L_m =\mathcal{O}_X(K_X+mA)$ satisfies property $N_p$ for any $m\geqslant n+1+p$. 
\end{theorem}

We keep notation as in the previous subsection, and we furthermore
define by analogy $\Gamma_{X_0 (r)} ^j(s)$ and 
$\Delta_{X_0 (r)}^j (s)$. Consider 

\[
F=L_m \boxtimes L_m \boxtimes \cdots \boxtimes L_m
\]

We show by double induction on $1\leqslant s\leqslant r$ that
\[
H^0(\Gamma_{X(r)}^1 (r), F)\rightarrow H^0(\Gamma^1 _{X(r)}(s) \cup \Delta_{X(r)}^{1,s+1}\cup \cdots \cup\Delta_{X(r)}^{1,r}, F)
\]

is surjective. The case $r=2$ and $s=1$ was done above,
whereas the case $r=2$ and $s=2$ is trivial.
Assume therefore that the result holds for any $1\leqslant s\leqslant r-1$, and let us prove it for $r$. We use descending induction on $s$. The case $s=r$ is obvious. Suppose therefore that $1\leqslant s<r$ and pick a section
\[
u\in H^0(\Gamma^1 _{X(r)}(s)\cup \Delta_{X(r)}^{1,s+1}\cup \cdots \cup \Delta_{X(r)}^{1,r}, F)
\]
By inductive hypothesis, there is a surjection
\[
H^0(\Gamma_{X(r)}^1 (s)
\cup \Gamma_{X(r)}^{1,s+2} \cup \cdots \cup \Gamma_{X(r)}^{1,r}, F)
\rightarrow
H^0(\Gamma_{X(r)}^1 (s)
\cup \Delta_{X(r)}^{1,s+2} \cup \cdots \cup \Delta_{X(r)}^{1,r}, F)
\]

By Theorem \ref{mainvanishing}, there is a surjection
\[
H^0(\Gamma_{X(r)} ^1 (r), F)\rightarrow H^0(\Gamma_{X(r)}^1 (s)
\cup \Gamma_{X(r)}^{1,s+2} \cup \cdots \cup \Gamma_{X(r)}^{1,r}, F)
\]

Therefore, we may assume that $u$ maps to zero in 
\[
H^0(\Gamma_{X(r)}^1 (s)
\cup \Delta_{X(r)}^{1,s+2} \cup \cdots \cup \Delta_{X(r)}^{1,r}, F)
\]

Consider the short exact sequence
\begin{align*}
0& \rightarrow H^0(\Gamma_{X(r)}^1 (s+1)
\cup \Delta_{X(r)}^{1,s+2} \cup \cdots \cup \Delta_{X(r)}^{1,r},F)
\rightarrow\\
& \rightarrow 
H^0(\Gamma_{X(r)}^1 (s)
\cup \Delta_{X(r)}^{1,s+2} \cup \cdots \cup \Delta_{X(r)}^{1,r},F)
\oplus
H^0(\Gamma_{X(r)} ^{1,s+1},F)\rightarrow \\
& \rightarrow H^0((\Gamma_{X(r)}^1 (s)
\cup \Delta_{X(r)}^{1,s+2} \cup \cdots \cup \Delta_{X(r)}^{1,r}) \cap 
\Gamma_{X(r)} ^{1,s+1},F)
\end{align*}

We may extend $u|_{\Delta_{X(r)}^{1,s+1}}$ to $\Gamma_{X(r)}^{1,s+1}$
via the Galois action as in Subsection \ref{subsectionnormality}.
By construction, this patches with
\[
u|_{\Gamma_{X(r)}^1 (s)
\cup \Delta_{X(r)}^{1,s+2} \cup \cdots \cup \Delta_{X(r)}^{1,r}}=0
\]

Therefore, it yields a section 
\[
\tilde{u}\in H^0(\Gamma_{X(r)}^1 (s+1)
\cup \Delta_{X(r)}^{1,s+2} \cup \cdots \cup \Delta_{X(r)}^{1,r},F)
\]

Finally, we may lift $\tilde{u}$ to $\Gamma_{X(r)}^1 (r)$ by 
inductive hypothesis.
Of course, an entirely analogous proof holds for $\Gamma_{X(r)} ^2(r)$.
By the particular case $s=1$, we have a surjective map
\[
H^0(\Gamma_{X(r)} ^j (r), F|_{\Gamma_{X(r)} ^j (r)})
\rightarrow
H^0(\Delta_{X(r)} ^j (r), F|_{\Delta_{X(r)} ^j (r)})
\]

Theorem \ref{maintheorem1} follows by combining the above surjection in the case $j=1$ with Theorem \ref{mainvanishing}.
By using the case $j=2$ instead, we get the following.

\begin{theorem}\label{maintheorem2}
Let $X$ be a smooth complex projective variety of dimension $n$ and let $A$ be an ample and basepoint free divisor.
Then the graded ring of sections of the line bundle
$L_m =\mathcal{O}_X(K_X+mA)$ is Koszul for any $m\geqslant n+2$.
\end{theorem}

Ein and Lazarsfeld showed that if $A$ is very ample then one may get a slightly stronger bound provided that $(X,A)\neq (\mathbb{P}^n, H)$. More precisely, they showed that $K_X+mA$ satisfies property $N_p$ for all $m\geqslant n+p$ if $p>0$. It is then natural to ask:

\begin{question}\label{quest2}
Let $X$ be a smooth complex projective variety of dimension $n$ and let
$A$ be an ample and basepoint free divisor. Suppose that
$(X,A)\neq (\mathbb{P}^n, H)$, where $H$ is a hyperplane section.
Does $\mathcal{O}_X(K_X+mA)$ satisfy property $N_p$ for $m\geqslant n+p$
if $p>0$?
\end{question}

\subsection{Examples.}\label{examples} We conclude by
discussing some
examples which show that the hypothesis of Theorem \ref{maintheorem1} 
may not be strengthened beyond Question \ref{quest2}.

Fix integers $n$, $d$ and $r$. 
Let $\alpha: X\rightarrow \mathbb{P}^n$ be the degree $d$
simple cyclic cover ramified along a general hypersurface of
degree $dr$. Let $H$ be a hyperplane section
and let $A=\alpha^* H$. We have
\[
K_X = (-(n+1)+(d-1)r)A
\]

and 
\[
\alpha_* \mathcal{O}_X = 
\mathcal{O}_{\mathbb{P}^n}\oplus \mathcal{O}_{\mathbb{P}^n}(-r)\oplus \cdots \oplus \mathcal{O}_{\mathbb{P}^n}(-(d-1)r)
\]

As usual, we denote $L_m=\mathcal{O}_X(K_X+mA)$.
By the above, we have
\[
L_m = \mathcal{O}_X(K_X+mA)=\mathcal{O}_X((m-n-1+(d-1)r)A)
\]

Now choose $m=n$, $d=2$ and $r=2$.
Then we get $L_n = \mathcal{O}_X(A)$,
\[
H^0(X,L_n)=H^0(\mathbb{P}^n, \mathcal{O}_{\mathbb{P}^n}(1))
\]
and
\[
H^0(X,L_n^{\otimes 2})=H^0(\mathbb{P}^n, \mathcal{O}_{\mathbb{P}^n}(2))\oplus H^0(\mathbb{P}^n, \mathcal{O}_{\mathbb{P}^n})
\]
Therefore $L_n$ is not projectively normal even when $X\neq\mathbb{P}^n$.

Suppose now that $|A|$ realizes $X$ as a double cover of $Y$ with 
$Y\neq \mathbb{P}^n$. It is an open question whether $\mathcal{O}_X(K_X+nA)$ is projectively normal or not. It was shown in \cite{purnachen} however that for Horikawa varieties $|K_X|$ realizes
$X$ as a double cover of a variety of minimal degree $Y$ and $nK_X$ is
not projectively normal. Therefore the bound cannot be lowered to $n-1$ even when $Y\neq \mathbb{P}^n$.

On the other hand, the bounds of Question \ref{quest1} are optimal
if one drops the assumption that $A$ is basepoint free, as already seen
in the case of curves \cite{greenlazarsfeld1}. Going up in dimension,
one may construct examples for which the bound on projective normality and $N_1$ is
optimal in the case of elliptic ruled surfaces (see \cite{purna1}).

Finally, going back to the construction for cyclic covers, choose $n=2$, $m=2$, $d=3$ and
$r=3$. It was shown in \cite[Example 5.2]{purna2} that 
$L_2 = \mathcal{O}_X (K_X+ 2A)$ satisfies property $N_0$ but not $N_1$.

\bibliography{biblio}
\bibliographystyle{alpha}

\end{document}